\documentclass[12pt,reqno]{article}

\usepackage[usenames]{color}
\usepackage{amssymb}
\usepackage{amscd}

\usepackage[colorlinks=true,
linkcolor=webgreen,
filecolor=webbrown,
citecolor=webgreen]{hyperref}

\definecolor{webgreen}{rgb}{0,.5,0}
\definecolor{webbrown}{rgb}{.6,0,0}

\usepackage{color}
\usepackage{fullpage}
\usepackage{float}

\usepackage{graphics,amsmath,amssymb}
\numberwithin{equation}{section}
\usepackage{amsthm}
\usepackage{amsfonts}
\usepackage{latexsym}

\begin{document}


\theoremstyle{plain}
\newtheorem{theorem}{Theorem}
\newtheorem{corollary}[theorem]{Corollary}
\newtheorem{lemma}{Lemma}
\newtheorem{proposition}{Proposition}
\theoremstyle{remark}
\newtheorem{remark}[theorem]{Remark}
\theoremstyle{plain}
\newtheorem*{example}{Examples}

\newcommand{\lrf}[1]{\left\lfloor #1\right\rfloor}

\begin{center}
\vskip 1cm{\LARGE\bf On a family of infinite series with reciprocal Catalan numbers \\
\vskip .11in }

\vskip 1cm

{\large

Kunle Adegoke \\
Department of Physics and Engineering Physics, \\ Obafemi Awolowo University, 220005 Ile-Ife, Nigeria \\
\href{mailto:adegoke00@gmail.com}{\tt adegoke00@gmail.com}

\vskip 0.2 in

Robert Frontczak\footnote{Statements and conclusions made in this article by R. F. are entirely those of the author. They do not necessarily reflect the views of LBBW.} \\
Landesbank Baden-W\"urttemberg (LBBW), Stuttgart,  Germany \\
\href{mailto:robert.frontczak@lbbw.de}{\tt robert.frontczak@lbbw.de}

\vskip 0.2 in

Taras Goy  \\
Faculty of Mathematics and Computer Science\\
Vasyl Stefanyk Precarpathian National University, Ivano-Frankivsk, Ukraine\\
\href{mailto:taras.goy@pnu.edu.ua}{\tt taras.goy@pnu.edu.ua}}
\end{center}

\vskip .2 in

\begin{abstract}
We study a certain family of infinite series with reciprocal Catalan numbers. We first evaluate two special candidates of the family in closed form, where we also present some Catalan-Fibonacci relations. Then we focus on the general properties of the family and prove explicit formulas, including two types of integral representations.
\end{abstract}

\section{Introduction and Motivation}

The famous Catalan numbers $C_n, n\geq 0,$ are defined by $C_n = \frac{1}{n+1}\binom {2n}{n}$.
They can be also expressed by the recursion

\begin{equation*}
C_n = \frac{2(2n-1)}{n+1} C_{n-1}, \quad C_0=1.
\end{equation*}

The generating function for $C_n$ is

\begin{equation*}
\sum_{n=0}^\infty C_n z^n = \frac{1-\sqrt{1-4z}}{2z}.
\end{equation*}

Catalan numbers have a long history and play an extraordinary role in combinatorics.
Excellent sources on these numbers are the books by Koshy \cite{koshy2}, Roman \cite{roman} and Stanley \cite{stanley}.
Some examples of recent work involving Catalan numbers and their generalizations include \cite{barry,chammam,chu,goy,kim,lin,qiguo1,qiguo2,zhang}.

This paper was inspired by a recent paper by Amdeberhan et~al. \cite{amde}, who studied the function
\begin{equation*}
f(z) = \sum_{n=0}^\infty \frac{z^n}{C_n}, \qquad |z| < 4,
\end{equation*}
which generates the reciprocals of Catalan numbers. They prove by several methods that
\begin{equation}\label{amde1}
f(z) = \frac{2(z+8)}{(4-z)^2} + \frac{24\sqrt{z} \arcsin(\sqrt{z}/2)}{(4-z)^{5/2}}.
\end{equation}
This expression also appears in \cite{yinqi} and an equivalent form is given in \cite{koshygao}.
The function $f(z)$ was also studied in the recent paper \cite{adfrgo}, where special attention was paid to series with Catalan-Fibonacci entries. The hypergeometric expression for $f(z)$ is
\begin{equation*}\label{amde2}
f(z) = {}_{2}F_{1}\Big(1,2;\frac12;\frac{z}{4}\Big),
\end{equation*}
where
\begin{equation*}
{}_{2}F_{1}(a,b;c;z) = \sum_{n=0}^\infty \frac{(a)_n (b)_n}{(c)_n}\frac{z^n}{n!}
\end{equation*}
and $(a)_n=a(a+1)\cdots (a+n-1), (a)_0=1,$ is the Pochhammer symbol. 

Inserting $z=2$ and $z=3$, respectively, in \eqref{amde1} yield the pretty evaluations
\begin{equation*}
\sum_{n=0}^\infty \frac{2^n}{C_n} = 5 + \frac{3\pi}{2} \quad \mbox{and} \quad \sum_{n=0}^\infty \frac{3^n}{C_n} = 22 + 8\sqrt{3}\pi,
\end{equation*}
which were stated in 2014 by Beckwith and Harbor as Problem 11765 in the American Mathematical Monthly \cite{beckwith} and solved by Abel \cite{abel}.

From the paper \cite{stewart21} (see also \cite{adfrgo}) we 
have the identities
\begin{equation*}
\sum_{n=1}^\infty \frac{F_{2n}}{C_n} = \frac{22}{5} + \frac{6(5+9\sqrt5)\pi}{125} \,\omega \quad \mbox{and} \quad
\sum_{n=0}^\infty \frac{L_{2n}}{C_n} = \frac{62}{5} + \frac{6(15+19\sqrt5)\pi}{125} \,\omega,
\end{equation*}
where $F_n$ and $L_n$ are the famous Fibonacci and Lucas numbers, respectively,  $\alpha=\frac{1+\sqrt{5}}{2}$ is the golden ratio and $\omega=\sqrt{\sqrt5\alpha}=\sqrt{2+\alpha}$. These numbers are defined for $n\geq 0$ by the recursions $F_{n+2} = F_{n+1} + F_n$ and $L_{n+2} = L_{n+1} + L_n$ with initial conditions $F_0 = 0, F_1 = 1$, $L_0=2$ and $L_1=1$, respectively. The Binet formulas are given by
\begin{equation*}
F_n = \frac{\alpha^n-\beta^n}{\alpha-\beta}, \qquad L_n = \alpha^n + \beta^n,
\end{equation*}
where $\beta=-\frac{1}{\alpha}=\frac{1-\sqrt{5}}{2}$. For negative subscripts we have 
\begin{equation*}
F_{-n}=(-1)^{n+1} F_n\quad \mbox{and} \quad L_{-n}=(-1)^n L_n.
\end{equation*}
See the book by Koshy \cite{koshy1} for more details.

Our purpose in this paper is to study, for each integer $m\geq 0$, the following family of series
\begin{equation*}
g_m(z) = \sum_{n=0}^\infty \frac{2^{2n} n^m}{2n+1}\frac{z^n}{C_n}.
\end{equation*}
These series converge for all $|z|<1$. We begin with evaluating the functions $g_0(z)$ and $g_1(z)$ explicitly
for some values of $z$, including Fibonacci and Lucas numbers. Then, focusing on $g_m(z)$,
we prove some explicit expressions for $g_m(z)$, including two integral representations.
\section{The functions $g_0(z)$ and $g_1(z)$}
Sprugnoli \cite{sprug} has shown that
\begin{equation}\label{sprug_id}
\sum_{n=0}^\infty \frac{2^{2n} z^{n + 1}}{(2n+1)\binom {2n}{n}} = \sqrt{\frac{z}{1-z}} \arctan \Big (\sqrt{\frac{z}{1-z}}\Big ),
\qquad |z|<1.
\end{equation}
Since $\arctan\Big(\!\sqrt\frac{z}{1-z}\Big) = \arcsin{(\sqrt{z})}$, the above identity could be stated equivalently using the arcsine function as in \eqref{amde1}. In this paper, however,
we have decided to work with the notation used by Sprugnoli. Differentiating both sides of \eqref{sprug_id}
with respect to $z$ gives
\begin{equation*}\label{id_g0}
g_0(z) = \frac{1}{2(1-z)} + \frac{1}{2(1-z)^2}\frac{\arctan \left({\sqrt {\frac{z}{{1 - z}}}} \right)}{\sqrt{\frac{z}{1 - z}}}, \qquad |z|<1.
\end{equation*}
Differentiating once more and multiplying by $z$ gives
\begin{equation}\label{id_g1}
g_1(z) = \frac{2z+1}{4(1-z)^ 2} + \frac{4z-1}{4(1-z)^3}\frac{\arctan \Big (\sqrt{\frac{z}{1-z}}\Big )}{\sqrt{\frac{z}{1-z}}}, \qquad |z|<1.
\end{equation}
The trigonometric versions of $g_0(z)$ and $g_1(z)$ are also useful, namely,
\begin{equation*} 
g_{t,0}(z) = \sum_{n=0}^\infty \frac{2^{2n} \sin^{2n} z}{(2n + 1)C_n} =
\frac{1}{2} \Big ( \frac{1}{\cos^2 z} + \frac{z}{\cos^3 z\sin z}\Big ) = \frac{1}{2\cos^2z}\left(1+\frac{2z}{\sin 2z}\right),
\end{equation*}


\begin{equation}\label{id_t_g1}
g_{t,1}(z) = \sum_{n=0}^\infty \frac{2^{2n} n \sin ^{2n} z}{(2n + 1)C_n }
= \frac{1}{4\cos^4z} \left(2-\cos2z+(1-2\cos 2z)\frac{2z}{\sin 2z} \right),
\end{equation}
both valid for $|z| < \frac{\pi}{2}$. 

At $z=\frac{\pi}{3}$, $z=\frac{\pi}{4}$ and $z=\frac{\pi}{6}$ function $g_{t,0}(z)$, respectively, gives
\begin{equation*}
\sum_{n=0}^\infty \frac{3^n}{(2n + 1)C_n} = 2 + \frac{{8\pi \sqrt 3}}{9},\qquad
\sum_{n=0}^\infty \frac{2^n}{(2n + 1)C_n} = 1 + \frac{\pi }{2},
\end{equation*}


\begin{equation*}
\sum_{n=0}^\infty \frac{1}{(2n + 1)C_n} = \frac{2}{3} + \frac{{4\pi \sqrt 3}}{27}.
\end{equation*}
We also have from \eqref{id_g1} or \eqref{id_t_g1}
\begin{equation*}
\sum_{n=1}^\infty \frac{n}{(2n+1) C_n} = \frac{2}{3},\quad\qquad 
\sum_{n=1}^\infty \frac{(-1)^n n}{(2n+1)C_n} = \frac{2}{25} \Bigl( 1-\frac{16}{\sqrt5}\ln\alpha \Bigr),
\end{equation*}

\begin{equation}\label{specialLuc}
\sum_{n=1}^\infty \frac{2^{n} n}{(2n+1) C_n} = 2 + \frac{\pi}{2},\quad\qquad \sum_{n=1}^\infty \frac{(-1)^n n 2^{n} }{(2n+1)C_n} = \frac{\sqrt3}{9} \ln \bigl(2-\sqrt3\bigr),
\end{equation}

\begin{equation*}
\sum_{n=1}^\infty \frac{3^{n} n}{(2n+1) C_n} = 10 + \frac{32 \pi}{3\sqrt{3}},\quad\qquad \sum_{n=1}^\infty \frac{(-1)^n n 3^n}{(2n+1)C_n} = -\frac{2}{49} \left(1-\frac{16}{\sqrt{21}} \ln \Bigl(\frac{5-\sqrt{21}}{2}\Bigr) \right).
\end{equation*}
To offer some evaluations of $g_0(z)$ involving Fibonacci and Lucas numbers, we need two lemmas.
\begin{lemma}[{\cite[Lemma 1]{adegoke_inv}}, see also {\cite[p.~271, identities (20)--(22)]{srivastava12}}] \label{lem.iqkulim}
We have
\begin{equation*}
\sin \left( {\frac{\pi }{{10}}} \right) = - \frac{\beta }{2}, \quad \sin \left( {\frac{{3\pi }}{{10}}} \right) = \frac{\alpha }{2},
\quad \cos \left( {\frac{\pi }{{10}}} \right) = \frac{{\sqrt {\alpha \sqrt 5 } }}{2}, \quad
\cos \left( {\frac{{3\pi }}{{10}}} \right) = \frac{1}{2} \sqrt{\frac{5}{\alpha\sqrt5}}.
\end{equation*}
\end{lemma}
\begin{theorem}
With $\omega=\sqrt{\sqrt5\alpha}=\sqrt{2+\alpha}$ we have for any integer $s$,

\begin{equation*}
5\sum_{n=0}^\infty \frac{F_{2n+s}}{(2n+1)C_n} = 2L_{s+1} + \frac{4\sqrt5\pi\,\omega}{25}\big((2+2\alpha)F_s + (4-\alpha)F_{s-1}\big)
\end{equation*}


\begin{equation*}
\sum_{n=0}^\infty \frac{L_{2n+s}}{(2n+1)C_n} = 2F_{s+1} + \frac{4\sqrt5\pi\,\omega}{25}\left(2F_s + \alpha F_{s-1}\right).
\end{equation*}

\end{theorem}
\begin{proof}
Set $z=\frac{3\pi}{10}$ in $g_{t,0}(z)$ and multiply through by $\alpha^s$, where $s$ is an arbitrary integer.
Using Lemma \ref{lem.iqkulim} yields
\begin{equation*}
\sum_{n=0}^\infty \frac{\alpha^{2n+s}}{(2n + 1)C_n} = \frac{2}{\sqrt5}\alpha^{s+1} + \frac{12\pi}{\sqrt{125\sqrt5}}\alpha^{s}\sqrt{\alpha}.
\end{equation*}
In like manner, set $z=\frac{\pi}{10}$ in $g_{t,0}(z)$ and multiply through by $\beta^s$ to obtain
\begin{equation*}
\sum_{n = 0}^\infty \frac{\beta^{2n + s}}{(2n + 1)C_n} = -\frac{2}{\sqrt5}\beta^{s+1} + \frac{4\pi}{\sqrt{125\sqrt5}}\frac{\beta^s}{\sqrt{\alpha}}.
\end{equation*}
The difference and sum of the above identities result in the identities
\begin{equation*}
\sum_{n = 0}^\infty  {\frac{{F_{2n + s}}}{{(2n + 1)C_n }}} = \frac{2}{5}L_{s + 1} + \frac{{4\pi }}{{25\sqrt {\alpha \sqrt 5 } }}
(\alpha L_{s + 2} + L_{s - 1} ),
\end{equation*}

\begin{equation*}
\sum_{n = 0}^\infty  {\frac{{L_{2n + s}}}{{(2n + 1)C_n }}} = 2F_{s + 1} + \frac{{4\pi }}{{\sqrt {125\alpha \sqrt 5 } }}(\alpha (F_{s + 1}
+ L_s ) + F_s + L_{s + 1} ).
\end{equation*}
The stated identities follow upon simplifications.
\end{proof}

As examples, we have with 
$\omega=\sqrt{\sqrt5\alpha}$\,:

\begin{equation*}
\sum_{n=1}^\infty \frac{F_{2n}}{(2n+1)C_n} = \frac{2}{5} + \frac{4(7\alpha-6)\pi\omega}{125},\qquad\quad \sum_{n=1}^\infty \frac{L_{2n}}{(2n+1)C_n} = \frac{4(\alpha+2)\pi\omega}{25},
\end{equation*}

\begin{equation*}
\sum_{n=1}^\infty \frac{F_{2n+1}}{(2n+1)C_n} = \frac{1}{5} + \frac{8(3\alpha+1)\pi\omega}{125},\qquad\quad \sum_{n=1}^\infty \frac{L_{2n+1}}{(2n+1)C_n} = 1 + \frac{8\sqrt5\pi\omega}{25},
\end{equation*}

\begin{equation*}
\sum_{n=1}^\infty \frac{F_{2n-2}}{(2n+1)C_n} = \frac{3}{5} + \frac{8(4\alpha-7)\pi\omega}{125},\qquad\quad
\sum_{n=1}^\infty \frac{L_{2n-2}}{(2n+1)C_n} = -1 + \frac{8(3-\alpha)\pi\omega}{25}.
\end{equation*}

Using the same idea for $g_{t,1}(z)$ we can prove the next theorem, those proof we omit.
\begin{theorem}
With $\omega=\sqrt{\sqrt5\alpha}=\sqrt{2+\alpha}$ we have for any integer $s$,

\begin{equation*}
\sum_{n=1}^\infty \frac{n F_{2n+s}}{(2n+1)C_n} = 2F_{s+1} + \frac{8}{5} F_{s}
+ \frac{8\sqrt5\pi\,\omega}{125}\left(\sqrt5\alpha F_{s+1} + 2F_{s} \right)
\end{equation*}

and

\begin{equation*}
\sum_{n=1}^\infty \frac{n L_{2n+s}}{(2n+1)C_n} = 2L_{s+1} + \frac{8}{5} L_{s} + \frac{8\sqrt5\pi\,\omega}{125}\left((6+\alpha)F_{s+1}+2\alpha^2F_{s}\right).
\end{equation*}

\end{theorem}

Another interesting example for an evaluation of $g_1(z)$ with Fibonacci (Lucas) entries is the following result.
\begin{theorem}\label{thm.xa8v8d3}
Let $r$ be an even integer and $s$ any integer. Then

\begin{align*}
\sum_{n = 1}^\infty  {\frac{{2^{2n} nF_{rn + s} }}{(2n + 1){L_r^n C_n }}}  &
= \left( {F_{3r + s}  + \frac{{L_r }}{2}F_{2r + s} } \right)\frac{{L_r }}{2}+\left( {L_r L_{2r + s}  - 4L_{3r + s} } \right)\frac{{L_r^2 }}{{4\sqrt 5 }}\arctan (\beta ^r )\\
&\quad\,\,+ \left( {4L_{3r + s}  - L_r L_{2r + s}  + (4F_{3r + s}  - L_r F_{2r + s} )\sqrt 5 } \right)\frac{{\pi L_r^2 }}{{16\sqrt 5 }},
\end{align*}

\begin{align*}
\sum_{n = 1}^\infty  {\frac{{2^{2n} nL_{rn + s} }}{(2n + 1){L_r^n C_n }}}  &
= \left( {L_{3r + s}  + \frac{{L_r }}{2}L_{2r + s} } \right)\frac{{L_r }}{2} + \left( {L_r F_{2r + s}  - 4F_{3r + s} } \right)\frac{{L_r^2\sqrt 5 }}{{4}}\arctan (\beta ^r )\\
&\quad\,\,+ \left( {4L_{3r + s}  - L_r L_{2r + s}  + (4F_{3r + s}  - L_r F_{2r + s} )\sqrt 5 } \right)\frac{{\pi L_r^2 }}{{16}}.
\end{align*}

\end{theorem}
\begin{proof}
First note that $1-\alpha^r/L_r=\beta^r/L_r$ and that if $r$ is an even integer, then

\begin{equation*}
\sqrt{\frac{\alpha^r/L_r}{1-\alpha^r/L_r}} = \sqrt{\frac{\alpha^r}{L_r-\alpha^r}} = \alpha^r.
\end{equation*}
Let $s$ be an arbitrary integer. Consider $\alpha ^s g_1 (\alpha ^r /L_r ) \mp \beta ^s g_1 (\beta ^r /L_r )$. We have
\begin{align*}
&\sum_{n = 1}^\infty  {\frac{{2^{2n} n(\alpha ^{rn +s}  \mp \beta ^{rn + s} )}}{{L_r^n (2n + 1)C_n }}}\\
&\qquad= (\alpha ^{3r + s}  \mp \beta ^{3r + s} )\frac{{L_r }}{2} + (\alpha ^{2r + s}  \mp \beta ^{2r + s} )\frac{{L_r^2 }}{4}
+ \left( {\frac{{4\alpha ^{2r + s} }}{{\beta ^r }} - \frac{{\alpha ^{r + s} }}{{\beta ^r }}L_r } \right)\frac{{L_r^2 }}{4}\arctan (\alpha ^r )\\
&\quad\qquad\,\mp \left( {\frac{{4\beta ^{2r + s} }}{{\alpha ^r }} - \frac{{\beta ^{r + s} }}{{\alpha ^r }}L_r } \right)\frac{{L_r^2 }}{4}\arctan (\beta ^r ),
\end{align*}

from which the stated identities in the theorem follow upon using the Binet formulas and the fact that $\arctan\alpha^r=\frac{\pi}{2}-\arctan\beta^r$ for any even integer $r$.
\end{proof}

As particular instances of Theorem \ref{thm.xa8v8d3}, we have, for even integer $r$,

\begin{equation*}
\sum_{n = 1}^\infty  {\frac{{2^{2n + 1} nF_{(n - 2)r} }}{{(2n + 1)L_r^n C_n }}}  = F_{2r}  + (L_r  + 2F_r \sqrt 5 )\frac{{\pi L_r^2 }}{{4\sqrt 5 }} - \frac{L_r^3} {\sqrt 5} \arctan \beta ^r,
\end{equation*}

\begin{equation}\label{Luc1}
\sum_{n = 1}^\infty  {\frac{{2^{2n} nL_{(n - 2)r} }}{{(2n + 1)L_r^n C_n }}}  = L_r^2  + (L_r  + 2F_r \sqrt 5 )\frac{{\pi L_r^2 }}{8} - F_r L_r^2 \sqrt 5 \arctan \beta ^r,
\end{equation}

\begin{align*}
\sum_{n = 1}^\infty  {\frac{{2^{2n} nF_{(n - 3)r} }}{{(2n + 1)L_r^n C_n }}} & = - \frac{1}{4}L_r^2 F_r + \left( {8 -  L_r^2  + F_{2r} \sqrt 5 } \right)\frac{{\pi L_r^2 \sqrt 5 }}{{80}}\\
&\quad\,+ \left( { L_r^2  - 8} \right)\frac{{L_r^2 \sqrt 5 \arctan \beta ^r }}{{20}},
\end{align*}

\begin{align}\label{Luc2}
\sum_{n = 1}^\infty  {\frac{{2^{2n} nL_{(n - 3)r} }}{{(2n + 1)L_r^n C_n }}}  &= \left( {2 + \frac{1}{2}L_r^2 } \right)\frac{L_r}2  + \left( {8 - ( - 1)^r L_r^2  + F_{2r} \sqrt 5 } \right)\frac{{\pi L_r^2 }}{{16}}\notag\\
&\quad\,- \frac{1}{4}L_r^3 F_r \sqrt 5 \arctan \beta ^r.
\end{align}

Note that both \eqref{Luc1} and \eqref{Luc2} give \eqref{specialLuc} when $r=0$.

\section{Integral expressions for $g_0(z)$ and $g_1(z)$}

\begin{lemma}\label{int}
For each integer $n\geq 0$, we have

\begin{equation*}
\int_{-1}^1 (1-x^2)^n dx = \frac{2^{2n+1}}{(2n+1) \binom {2n}{n}}.
\end{equation*}
\end{lemma}
The result is known. It can be proved easily using integration by parts.
It is, however, a special case of the more interesting fact \cite[p.~52]{koshy2} that
\begin{equation*}
\int_a^b(x-a)^n(b-x)^n dx = 2\cdot\frac{2\cdot4\cdot6\cdots(2n)}{3\cdot5\cdot7\cdots(2n+1)}\cdot\left(\frac{b-a}{2}\right)^{2n+1}, \qquad n\geq1.
\end{equation*}
\begin{theorem}\label{int_g01}
We have
\begin{equation*}\label{int_id_g0}
g_0(z) = \frac{1}{2} \int_{-1}^1 \frac{1}{(1-z(1-x^2))^2} dx,
\end{equation*}


\begin{equation*}\label{int_id_g1}
g_1(z) = z \int_{-1}^1 \frac{1-x^2}{(1-z(1-x^2))^3} dx.
\end{equation*}
\end{theorem}
\begin{proof}
From the geometric series and the above lemma we deduce that for all $|z|<1$
\begin{equation*}
\sum_{n=0}^\infty \frac{2^{2n} z^{n+1}}{(2n+1)\binom {2n}{n}} = \frac{1}{2} \int_{-1}^1 \frac{z}{1-z(1-x^2)} dx.
\end{equation*}
Differentiating produces the first equation. To obtain the second equation,
perform the operation $z (d/dz) g_0(z)$ and the proof is completed.
\end{proof}
It is interesting to compare the integral expressions for $g_0(z)$ and $g_1(z)$ with that for $f(z)$.
This expression is not stated explicitly in \cite{amde} but can be derived as follows:
\begin{align*}
\sum_{n=1}^\infty \frac{z^n}{C_n} & = \sum_{n=1}^\infty \frac{(n+1)}{\binom {2n}{n}} z^n \\
& = \sum_{n=1}^\infty \frac{n(n+1) \Gamma(n) \Gamma(n+1)}{\Gamma(2n+1)} z^n \\
& = \sum_{n=1}^\infty n(n+1) B(n,n+1) z^n \\
& = \sum_{n=1}^\infty n(n+1) \int_0^1 (1-x)^n x^{n-1} dx\, z^n \\
& = 2z \int_0^1 \frac{1-x}{(1-zx(1-x))^3} dx,
\end{align*}
where in the derivation $\Gamma(x)$ is the gamma function and $B(a,b)$ is the beta function
\begin{equation*}
B(a,b) = \int_0^1 (1-x)^{a-1} x^{b-1} dx = \frac{\Gamma(a) \Gamma(b)}{\Gamma(a+b)}.
\end{equation*}
This proves that
\begin{equation*}
f(z) = 1 + 2z \int_0^1 \frac{1-x}{(1-zx(1-x))^3} dx.
\end{equation*}
Finally, it is worth mentioning that by the binomial theorem the integral in Lemma \ref{int} can be also evaluated as
\begin{equation*}
\int_{-1}^1 (1-x^2)^n dx = \int_{-1}^1 \sum_{k=0}^n \binom {n}{k} (-1)^k x^{2k} dx = 2 \sum_{k=0}^n \binom {n}{k} \frac{(-1)^k}{2k+1}.
\end{equation*}

Hence,

\begin{align*}
\sum_{n=0}^\infty \frac{2^{2n} z^{n}}{(2n+1)\binom {2n}{n}} & =  \sum_{n=0}^\infty \sum_{k=0}^n \binom {n}{k} (-1)^k \frac{z^n}{2k+1} \\
& =  \sum_{k=0}^\infty \frac{(-1)^k}{2k+1} \sum_{n=k}^\infty \binom {n}{k} z^n \\
& = \sum_{k=0}^\infty \frac{(-1)^k z^k}{2k+1} \sum_{n=0}^\infty \binom {n+k}{k} z^n \\
& = \sum_{k=0}^\infty \frac{(-1)^k z^k}{2k+1} \sum_{n=0}^\infty \binom {n+k}{n} z^n \\
& = \sum_{k=0}^\infty \frac{(-1)^k z^k}{2k+1} \frac{1}{(1-z)^{k+1}}.
\end{align*}

This shows that

\begin{equation*}
\sum_{n=0}^\infty \frac{2^{2n} z^{n+1}}{(2n+1)\binom {2n}{n}} = \sum_{n=0}^\infty \frac{(-1)^n}{2n+1} \Big (\frac{z}{1-z}\Big )^{n+1},
\end{equation*}

and provides a new proof of Sprugnoli's identity \eqref{sprug_id}. 
Moreover,

\begin{equation*}
g_0(z) = \frac{1}{(1-z)^2} \sum_{n=0}^\infty \frac{(-1)^n(n+1)}{2n+1} \Big ( \frac{z}{1-z}\Big )^n,
\end{equation*}


\begin{equation*}
g_1(z) = \frac{1}{(1-z)^3} \sum_{n=0}^\infty \frac{(-1)^n(n+1)}{2n+1}(2z+n)\Big(\frac{z}{1-z}\Big )^n.
\end{equation*}

\section{Some general properties of $g_m(z)$}

In this section, we present some general properties of $g_m(z)$ which is defined by

\begin{equation*}
g_m(z) = \sum_{n=0}^\infty \frac{2^{2n} n^m}{2n+1}\frac{z^n}{C_n}, \qquad |z|<1,
\end{equation*}

with $m\geq 0$ being an integer. As  we have the following result:
\begin{theorem}\label{main1_gm}
For each $m\geq 0$ and all $|z|<1$, $g_m(z)$ possesses the representation

\begin{equation}\label{gm_id1}
g_m(z) = \frac{P_{1,m}(z)}{(1-z)^{m+1}} + \frac{P_{2,m}(z)}{(1-z)^{m+2}} \frac{\arctan \Big (\sqrt{\frac{z}{1-z}}\Big )}{\sqrt{\frac{z}{1-z}}},
\end{equation}

where $P_{1,m}(z)$ and $P_{2,m}(z)$ are polynomials in $z$ of degree $m$ with rational coefficients.
Moreover, for $m\geq 1$ the polynomials $P_{1,m}(z)$ and $P_{2,m}(z)$ can be expressed recursively according to

\begin{equation}\label{rec1}
P_{1,m}(z)  =  z(1-z)\frac{d}{dz} P_{1,m-1}(z) + m z P_{1,m-1}(z) + \frac{1}{2} P_{2,m-1}(z) \end{equation}

and

\begin{equation}\label{rec2}
P_{2,m}(z)  =  z(1-z)\frac{d}{dz} P_{2,m-1}(z) + (m+1)z P_{2,m-1}(z) - \frac{1}{2} P_{2,m-1}(z), 
\end{equation}

with $P_{1,0}(z)=P_{2,0}(z)=\frac12$.
\end{theorem}
\begin{proof}
The proof of the representation \eqref{gm_id1} is easy using induction on $m$ taking into account
$g_{m+1}(z)=z(d/dz)\,g_{m}(z)=(z d/dz)^m g_0(z)$ and the identity

\begin{equation*}
\frac{d}{dz} \frac{\arctan \Big (\sqrt{\frac{z}{1-z}}\Big )}{\sqrt{\frac{z}{1-z}}} = \frac{1}{2z} - \frac{1}{2z(1-z)}
\frac{\arctan \Big (\sqrt{\frac{z}{1-z}}\Big )}{\sqrt{\frac{z}{1-z}}}.
\end{equation*}

The recursive expressions for $P_{1,m}(z)$ and $P_{2,m}(z)$ follow from the proof as a by-product.
\end{proof}

We mention that the coupled recursions \eqref{rec1} and \eqref{rec2} can be solved explicitly, but the closed forms seem not to shed enough light on their general structure. Nevertheless, we can prove the following expressions:
\begin{proposition}
For each $m$,

\begin{equation*}
P_{1,m}(z) = \frac{1}{2} m! z^m + \sum_{j=0}^{m-1} \binom {m}{j} j! z^j \Big ( z(1-z) \frac{d}{dz} P_{1,m-(j+1)}(z)
+ \frac{1}{2} P_{2,m-(j+1)}(z)\Big ),
\end{equation*}


\begin{equation*}
P_{2,m}(z) = \frac{1}{2} \prod_{j=1}^m \Big ((j+1)z-\frac{1}{2}\Big ) + z(1-z) \sum_{j=1}^m \frac{d}{dz} P_{2,m-j}(z)
\prod_{k=2}^j \Big ((m+3-k)z-\frac{1}{2}\Big ),
\end{equation*}

where the empty product is one and the empty sum is zero.
\end{proposition}
\begin{proof}
We can use induction on $m$ to prove both formulas. For $m=0$ the statements are true. The inductive step for $P_{1,m}(z)$ is

\begin{align*}
 & P_{1,m+1}(z) = (m+1) z P_{1,m}(z) + z(1-z)\frac{d}{dz} P_{1,m}(z) + \frac{1}{2} P_{2,m}(z) \\
 & =  \frac{1}{2} (m+1)! z^{m+1} + \sum_{j=0}^{m-1} \binom {m}{j} j! (m+1) z^{j+1} \Big ( z(1-z) \frac{d}{dz} P_{1,m-(j+1)}(z)
+ \frac{1}{2} P_{2,m-(j+1)}(z)\Big ) \\
& \quad + z(1-z)\frac{d}{dz} P_{1,m}(z) + \frac{1}{2} P_{2,m}(z) \\
  & = \frac{1}{2} (m+1)! z^{m+1} +\! \sum_{j=-1}^{m-1} \!\binom {m+1}{j+1} (j+1)! z^{j+1} \Big ( z(1-z) \frac{d}{dz} P_{1,m-(j+1)}(z) + \frac{1}{2} P_{2,m-(j+1)}(z)\Big ) \\
& = \frac{1}{2} (m+1)! z^{m+1} + \sum_{j=0}^{m} \binom {m+1}{j} j! z^{j} \Big ( z(1-z) \frac{d}{dz} P_{1,m-j}(z)
+ \frac{1}{2} P_{2,m-j}(z)\Big ).
\end{align*}

Similarly, the inductive proof for $P_{2,m}(z)$ is accomplished according to

\begin{align*}
& P_{2,m+1}(z)  = \Big ((m+2)z-\frac{1}{2}\Big )P_{2,m}(z) + z(1-z)\frac{d}{dz} P_{2,m}(z) \\
\quad& = \Big ((m+2)z - \frac{1}{2}\Big )
\left(\frac{1}{2} \prod_{j=1}^m \Big ((j+1)z-\frac{1}{2}\Big )\right. \\
\quad&\quad\left. \,+ z(1-z) \sum_{j=1}^m \frac{d}{dz} P_{2,m-j}(z) \prod_{k=2}^j \Big ((m+3-k)z-\frac{1}{2}\Big )\right) + z(1-z)\frac{d}{dz} P_{2,m}(z)\\
\quad& = \frac{1}{2} \prod_{j=1}^{m+1} \Big ((j+1)z-\frac{1}{2}\Big ) + z(1-z) \left( \frac{d}{dz} P_{2,m}(z) \right.\\
\quad& \quad\left. + \sum_{j=1}^m \frac{d}{dz} P_{2,m-j}(z) \Big ((m+2)z-\frac{1}{2}\Big ) \prod_{k=2}^j \Big ((m+3-k)z-\frac{1}{2}\Big )\right) \\
\quad& = \frac{1}{2} \prod_{j=1}^{m+1} \Big ((j+1)z-\frac{1}{2}\Big ) + z(1-z) \sum_{j=0}^m \frac{d}{dz} P_{2,m-j}(z)
 \prod_{k=1}^j \Big ((m+3-k)z-\frac{1}{2}\Big )\\
\quad& = \frac{1}{2} \prod_{j=1}^{m+1} \Big ((j+1)z-\frac{1}{2}\Big ) + z(1-z) \sum_{j=1}^{m+1} \frac{d}{dz} P_{2,m-(j-1)}(z)
 \prod_{k=1}^{j-1} \Big ((m+3-k)z-\frac{1}{2}\Big )\\
\quad& = \frac{1}{2} \prod_{j=1}^{m+1} \Big ((j+1)z-\frac{1}{2}\Big ) + z(1-z) \sum_{j=1}^{m+1} \frac{d}{dz} P_{2,m+1-j}(z)
 \prod_{k=2}^{j} \Big ((m+4-k)z-\frac{1}{2}\Big ).
\end{align*}

\end{proof}

Applying Theorem \ref{main1_gm} in the case $m=2$ yields

\begin{equation*}
\sum_{n=1}^\infty \frac{2^{2n} n^2}{2n+1}\frac{z^n}{C_n} = \frac{4z^2+12z-1}{8(1-z)^3}
+ \frac{16z^2-2z+1}{8(1-z)^4} \frac{\arctan \Big (\sqrt{\frac{z}{1-z}}\Big )}{\sqrt{\frac{z}{1-z}}},
\end{equation*}

from which we obtain

\begin{equation*}
\sum_{n=1}^\infty \frac{n^2}{(2n+1) C_n} = \frac{2}{3} + \frac{8\sqrt{3}\pi}{81},
\end{equation*}

\begin{equation*}
\sum_{n=1}^\infty \frac{2^{n} n^2}{(2n+1) C_n} = 6 + 2\pi,\qquad 
\sum_{n=1}^\infty \frac{3^{n} n^2}{(2n+1) C_n} = 82 +  \frac{272\sqrt{3}\pi}{9}.
\end{equation*}

\begin{corollary}
For each $m\geq 0$, the sum
$
%
\sum_{n=0}^\infty \frac{2^{n} n^m}{(2n+1) C_n}
%
$
can be expressed in the form $a+b\pi$, with $a$ and $b$ rational. The sums
%
$
\sum_{n=0}^\infty \frac{n^m}{(2n+1) C_n}$ and $\sum_{n=0}^\infty \frac{3^{n} n^m}{(2n+1) C_n}
$
%
allow the same representation but with $b$ irrational.
\end{corollary}
\begin{theorem}\label{main2_gm}
For each $m\geq 0$ and all $|z|<1$ $g_m(z)$, possesses the integral representation

\begin{equation*}\label{gm_id2}
g_m(z) = \int_{-1}^1 \frac{Q_m(z)}{(1-z(1-x^2))^{m+2}} dx,
\end{equation*}

where $Q_{m}(z)$ is a polynomial in $z$ of degree $m$ given by

\begin{align}\label{qm_closed}
Q_{m}(z) & = \frac{1}{2} (m+1)! (az)^m + (1-az) z \frac{d}{dz} Q_{m-1}(z) \nonumber \\
& \qquad\qquad + (1-az) z \sum_{j=1}^{m-1} \binom {m+1}{j} j! (az)^j \frac{d}{dz} Q_{m-(j+1)}(z),
\end{align}

with $Q_{0}(z)=\frac12$ and where we have set $a=1-x^2$.
\end{theorem}
\begin{proof}
We prove the claim by induction on $m$. Since

\begin{equation*}
g_0(z) = \frac{1}{2} \int_{-1}^1 \frac{1}{(1-z(1-x^2))^{2}} dx,
\end{equation*}

the statement is true for $m=0$. Now, assuming it is true for a fixed $m>0$, we can proceed with

\begin{equation*}
g_{m+1}(z) = z \frac{d}{dz}g_m(z) = z \int_{-1}^1 \frac{\frac{d}{dz}Q_m(z) (1-az)+(m+2)a Q_m(z)}{(1-z(1-x^2))^{m+3}} dx.
\end{equation*}

This gives the recursion

\begin{equation*}\label{qm_rec}
Q_{m}(z) = (m+1) a z Q_{m-1}(z) + z (1-az) \frac{d}{dz} Q_{m-1}(z), \quad m\geq 1.
\end{equation*}

This recursion can be solved by standard methods to give \eqref{qm_closed}. Alternatively, one can prove
\eqref{qm_closed} directly by induction on $m$ using \eqref{qm_rec}.
\end{proof}

\section{Another integral expression for $g_m(z)$ using Mellin transform}
\begin{lemma}\label{Mellin_lem}
For integers $m,n\geq 0$, we have

\begin{equation*}
\sum_{j=0}^m (-1)^j S(m+1,m+1-j) (n+m-j)! = n^m n!,
\end{equation*}

where $S(n,k)$ are the Stirling numbers of the second kind, defined by $S(0,0)=1$, $S(n,n)=S(n,1)=1 \, (n\geq 1),$ and

\begin{equation*}
S(n,k) = \frac{1}{k!}\sum_{s=0}^k (-1)^s \binom {k}{s} (k-s)^n.
\end{equation*}

\end{lemma}
\begin{proof}
Consider the known representation

\begin{equation*}
x^m=\sum_{j=0}^{m}\binom{x}{j}S(m,j)j!.
\end{equation*}

Let $x=-n$. Using
$\binom{s}{j}=(-1)^j\binom{-s+j-1}{j}$
we have

\begin{align*}
(-1)^{m+1} n^{m+1} &= \sum_{j=0}^{m+1} \binom{-n}{j} S(m+1,j)j! \\
&= \sum_{j=0}^{m+1}(-1)^j \binom{n+j-1}{j} S(m+1,j)j! \\
&= \sum_{j=0}^{m+1}(-1)^j \frac{(n+j-1)!}{(n-1)!} S(m+1,j).
\end{align*}
Thus, by reindexing the summation
\begin{align*}
(-1)^{m+1} n^{m+1}(n-1)! &= \sum_{j=0}^{m+1} (-1)^{m-j+1} S(m+1,m+1-j) (n+m-j)! \\
&= \sum_{j=0}^{m} (-1)^{m-j+1} S(m+1,m+1-j) (n+m-j)!
\end{align*}
as $S(n,0)=0, n\geq 1$.
\end{proof}
\begin{theorem}\label{main3_gm}
The function $g_m(z)$ possesses the integral representation

\begin{equation*}\label{gm_id3}
g_m(z) = \frac{1}{\sqrt{z}} \sum_{j=0}^m (-1)^j S(m+1,m+1-j)\int_0^\infty x^{(m-j)/2}K_{j+1-m}(2\sqrt{x})\sinh (2\sqrt{xz}) dx,
\end{equation*}

where $K_v(x)$ is the modified Bessel function of the second kind, which can be defined by

\begin{equation*}
K_v(x) = \int_0^\infty \cosh(vt) e^{-x\cosh t} dt \quad (x>0).
\end{equation*}

\end{theorem}
\begin{proof}
The proof bases on ideas developed in \cite{amde}.
Recall that the Mellin transform of a real valued function $f(x)$ on $(0,\infty)$
is defined by the integral \cite{deb}

\begin{equation*}
M[f(x)](s) = \int_0^\infty x^{s-1} f(x) dx.
\end{equation*}

The gamma function $\Gamma(n)$ can be interpreted as $M[e^{-x}](n)$ and thus 
\begin{equation*}
M[xe^{-x}](n+1)=(n+1)!.
\end{equation*}

Since

\begin{equation*}
g_m(z) = \sum_{n=0}^\infty \frac{(4z)^n}{(2n+1)!} n^m (n+1)! n!,
\end{equation*}

we want to find a function $f_m(x)$ such that

\begin{equation*}
n^m n! (n+1)! = M[f_m(x)] M[xe^{-x}](n+1).
\end{equation*}

By Lemma \ref{Mellin_lem} it follows that such a function is

\begin{equation*}
f_m(x) = \sum_{j=0}^m (-1)^j S(m+1,m+1-j) x^{m-j} e^{-x}.
\end{equation*}

Now, we are going to apply the Mellin convolution theorem

\begin{equation*}
M[f_1(x)] M[f_2(x)](s) = M[F(x)](s)
\end{equation*}

with

\begin{equation*}
F(x) = \int_0^\infty f_1(x_1) f_2\Big (\frac{x}{x_1}\Big )\frac{dx_1}{x_1}.
\end{equation*}

In our case $F(x)$ equals

\begin{align*}
F(x) & = \int_0^\infty \Big ( \sum_{j=0}^m (-1)^j S(m+1,m+1-j) x_1^{m-j} e^{-x_1} \Big ) \frac{x}{x_1} e^{-x/x_1} \frac{dx_1}{x_1} \\
& = \sum_{j=0}^m (-1)^j S(m+1,m+1-j) \int_0^\infty x \frac{e^{-(x_1+x/x_1)}}{x_1^{j+2-m}} dx_1 \\
& = 2 \sum_{j=0}^m (-1)^j S(m+1,m+1-j) x^{(m+1-j)/2} K_{j+1-m}(2\sqrt{x}),
\end{align*}

where the following representation for the modified Bessel function of the second kind was used (\cite{amde})

\begin{equation*}
K_v(z) = \frac{1}{2} \Big (\frac{z}{2}\Big )^v \int_0^\infty e^{-(t+z^2/4t)}\frac{dt}{t^{v+1}}.
\end{equation*}

Finally, we calculate

\begin{align*}
g_m(z) & = \sum_{n=0}^\infty \frac{(4z)^n}{(2n+1)!} n^m (n+1)! n! \\
& =  \sum_{n=0}^\infty \frac{(4z)^n}{(2n+1)!} M[F(x)](n+1) \\
& =  \int_0^\infty F(x) \sum_{n=0}^\infty \frac{(4xz)^n}{(2n+1)!} dx \\
& =  \int_0^\infty F(x) \frac{\sinh(2\sqrt{xz})}{2\sqrt{xz}} dx.
\end{align*}

\end{proof}
Two special cases of the representation are

\begin{equation*}
g_0(z) = \frac{1}{\sqrt{z}}\int_0^\infty \sinh (2\sqrt{xz}) K_1(2\sqrt{x}) dx
\end{equation*}

and

\begin{equation*}
g_1(z) = \frac{1}{\sqrt{z}}\int_0^\infty \sinh (2\sqrt{xz}) \Big (\sqrt{x} K_0(2\sqrt{x}) - K_1(2\sqrt{x})\Big ) dx.
\end{equation*}

\section{Concluding comments}

In this paper, we have studied an interesting family of infinite series involving Catalan numbers.
In particular, we have evaluated these series for special arguments and provided characterizations.
Before closing we remark that Sprugnoli's identity \eqref{sprug_id} can be integrated resulting in the identity

\begin{equation*}
\sum_{n=0}^\infty \frac{2^{2n}}{(2n+1)(n+1)(n+2)}\frac{z^{n}}{C_n} = \frac{1}{2z} + \frac{1}{2z^2} \arctan^2 \Big (\sqrt{\frac{z}{1-z}}\Big )
- \frac{\arctan \Big (\sqrt{\frac{z}{1-z}}\Big )}{z\sqrt{\frac{z}{1-z}}},\! \quad |z|<1,
\end{equation*}

which contains the evaluation

\begin{equation*}
\sum_{n=0}^\infty \frac{2^{n}}{(2n+1)(n+1)(n+2)C_n} =  \frac{\pi^2}{8} - \frac{\pi}{2} +1
\end{equation*}

as a special instance.




\end{document}